\documentclass{amsproc}
\usepackage{amssymb}
\usepackage{amsfonts}

\theoremstyle{plain}

\newtheorem{corollary}{Corollary}

\newtheorem{lemma}{Lemma}

\newtheorem{theorem}{Theorem}
\numberwithin{equation}{section}

\begin{document}
\title[Spherical multipliers]{Transferring spherical multipliers on compact
symmetric spaces}
\author{Sanjiv Kumar Gupta}
\address{Dept. of Mathematics and Statistics\\
Sultan Qaboos University\\
P.O.Box 36 Al Khodh 123\\
Sultanate of Oman}
\email{gupta@squ.edu.om}
\author{Kathryn E. Hare}
\address{Dept. of Pure Mathematics\\
University of Waterloo\\
Waterloo, Ont.,~Canada\\
N2L 3G1}
\email{kehare@uwaterloo.ca}
\thanks{This research is supported in part by NSERC 2016-03719 and by Sultan
Qaboos University. The first author thanks the University of Waterloo for
their hospitality when some of this research was done.}
\thanks{This paper is in final form and no version of it will be submitted
for publication elsewhere.}

\begin{abstract}
We prove a two-sided transference theorem between $L^{p}$ spherical
multipliers on the compact symmetric space $U/K$ and $L^{p}$ multipliers on
the vector space $i\mathfrak{p},$ where the Lie algebra of $U$ has Cartan
decomposition $\mathfrak{k\oplus }i\mathfrak{p}$. This generalizes the
classic theorem transference theorem of deLeeuw relating multipliers on $%
L^{p}(\mathbb{T)}$ and $L^{p}(\mathbb{R)}$.
\end{abstract}

\maketitle

\section{\protect\bigskip Introduction}

Let $G$ be any non-discrete, locally compact, unimodular group and let $%
L^{p}(G)$ denote the space of $p$-integrable functions on $G$ with respect
to the Haar measure. A bounded linear operator $T:L^{p}(G)\rightarrow
L^{p}(G)$ is said to be a multiplier (on $L^{p}(G)$) if $T$ commutes with
translations on $G$.

It is well known that a bounded linear map $T$ is a multiplier on $L^{p}(%
\mathbb{R)}$ precisely when there is a bounded measurable function $m$ on $%
\mathbb{R}$ such that $\widehat{Tf}=m\widehat{f}$ for all $f\in L^{p}\bigcap
L^{2}(\mathbb{R)}$. Similarly, a bounded linear map $T$ is a multiplier on $%
L^{p}(\mathbb{T)}$ if and only if there is a bounded function $m$ on $%
\mathbb{Z}$ such that $\widehat{Tf}=m\widehat{f}$ for all $f\in L^{p}\bigcap
L^{2}(\mathbb{T)}$. To emphasize the association with $m,$ we denote the
operator by $T_{m}$.

In 1965, deLeeuw proved two remarkable facts relating the multipliers of $%
L^{p}(\mathbb{R)}$ to those of $L^{p}(\mathbb{T)}$.

\begin{theorem}
\label{deL} Let $m$ be a uniformly continuous function on $\mathbb{R}$ and
for $\varepsilon >0$ let $m_{\varepsilon }$ be its restriction to $\mathbb{Z}%
/\varepsilon \subseteq \mathbb{R}$, which we identify with $\mathbb{Z}$.

(1) If $T_{m}$ is a multiplier on $L^{p}(\mathbb{R)}$, then $%
T_{m_{\varepsilon }}$ is a multiplier on $L^{p}(\mathbb{T)}$ and $\left\Vert
T_{m_{\varepsilon }}\right\Vert _{p,p}\leq \left\Vert T_{m}\right\Vert
_{p,p} $ where $\left\Vert \cdot \right\Vert _{p,p}$ denotes the operator
norm.

(2) If the operators $T_{m_{\varepsilon }}$ are uniformly bounded on $L^{p}(%
\mathbb{T)}$, then $T_{m}$ is a bounded linear operator on $L^{p}(\mathbb{R)}
$ with $\left\Vert T_{m}\right\Vert _{p,p}\leq \sup_{\varepsilon }\left\Vert
T_{m_{\varepsilon }}\right\Vert _{p,p}$.
\end{theorem}

These two elegant theorems became the prototype for a number of
\textquotedblleft transference\textquotedblright\ results, where the
boundedness of a multiplier operator (for instance, a convolution kernel on
a group) may be checked on a different, hopefully simpler, group. Finding
analogues in various settings continues to be of interest today, c.f. \cite%
{CP}.\texttt{\ }

This theme was taken up in the context of non-commutative harmonic analysis,
first by Coifman and Weiss \cite{CW1}, \cite{CW2}, who proved Marcinkiewicz
type results for $SU(2)$ and other Lie groups, replacing the quotient map: $%
\mathbb{R\rightarrow T}$ by the mapping $X\rightarrow \exp (X)$ from the Lie
algebra to its Lie group. Rubin \cite{Ru} used similar ideas in the context
of $SO(3)$ and the Euclidean motion group $M(2)$. The approach was next
pursued by Dooley and others, who showed that the notion of a contraction,
or a continuous deformation of Lie groups, was the key underlying idea: The
Lie group $G_{2}$ is said to be a contraction of the Lie group $G_{1}$ if
there is a family $(\pi _{\varepsilon })_{\varepsilon >0}$ of local
diffeomorphisms $\pi _{\varepsilon }:G_{2}\rightarrow G_{1},$ which are
approximate homomorphisms in the sense that $\pi _{\varepsilon
}(x)\rightarrow e$ as $\varepsilon \rightarrow 0$ and $\pi _{\varepsilon
}^{-1}(\pi _{\varepsilon }(x)\pi _{\varepsilon }(y))\rightarrow xy$ as $%
\varepsilon \rightarrow 0$. This notion generalizes the
homomorphism/dilation relationship between $\mathbb{R}$ and $\mathbb{T}$.
For example, if we have the Cartan decomposition $\mathfrak{u=k}\oplus i%
\mathfrak{p}$ of a Riemannian compact symmetric pair $(U,K)$, then the
Cartan motion group $\mathfrak{p}\rtimes K$ is a contraction of $U$ by the
maps $\pi _{\varepsilon }(X,k)=k\exp (\varepsilon X)$. In \cite{DG}, for
example, a version of the second part of deLeeuw's Theorem was proved in
this setting. Further results in this spirit were also given in \cite{D1}, 
\cite{DGu} - \cite{DRi} and \cite{RRu}.

The results of \cite{DG} can also be viewed as generalizations of the
results of Stanton in \cite{St} where a version of (2) was proven for the
transference of spherical multipliers on $U/K$ to $Ad(K)$-invariant
multipliers on $\mathfrak{p}$.

Subsequently, in \cite{D2} a version of the first part of deLeeuw's theorem,
both for the Cartan motion group contraction and the Coifman-Weiss
contraction of $U$ to $\mathfrak{u}$, was given. Unfortunately, the versions
of (1) proven in \cite{D2} no longer gave an exact converse of the version
of (2) from \cite{CW1}, \cite{DG}, \cite{St} etc. Thus an important open
question remains to find a suitable version of the \textquotedblleft
restriction\textquotedblright\ for which analogues of both (1) and (2)
(simultaneously) hold.

In this article, we will find a suitable version of the restriction for
which both directions of deLeeuw's theorem hold in the case of a contraction
of $U/K$ to $i\mathfrak{p}$ for compact symmetric spaces $U/K.$

\section{Harmonic Analysis of Symmetric Spaces}

\subsection{Symmetric spaces notation}

Let $U$ be a compact, simply connected, semisimple Lie group and suppose $%
\theta $ is an involution of $U$. The set of fixed points under $\theta ,$
denoted $K$, is a compact, connected subgroup of $U,$ and the quotient
space, $U/K$, is known as a compact symmetric space. We will let $\pi
:U\rightarrow U/K$ denote the quotient map and given $x\in U$, we let $%
\overline{x}=\pi (x)$ denote the coset $xK$.

The involution $\theta $ induces an involution of $\mathfrak{u}$, the Lie
algebra of $U$, which we also denote by $\theta $. Let $\mathfrak{k}$ and $i%
\mathfrak{p}$ denote the $\pm 1$ eigenspaces of $\theta $ respectively. The
decomposition $\mathfrak{u=k\oplus }i\mathfrak{p}$ is known as the Cartan
decomposition.

Let $\mathfrak{g}^{\mathbb{C}}$ denote the complexification of $\mathfrak{u}$
and let $\mathfrak{g}_{0}=\mathfrak{k}\mathfrak{\oplus p}$. Fix a maximal
abelian subspace $\mathfrak{a}$ of $\mathfrak{p}$ and choose a Cartan
subalgebra $\mathfrak{h}$ of $\mathfrak{g}_{0}$ containing $\mathfrak{a}$.
Let $\mathfrak{h}^{\mathbb{C}}$ denote its complexification and let $\Sigma $
denote the set of roots of $\mathfrak{g}^{\mathbb{C}}$ with respect to $%
\mathfrak{h}^{\mathbb{C}}$. Denote by $\Sigma ^{+}$ the positive roots and
let $\Phi ^{+}$ be given by%
\begin{equation*}
\Phi ^{+}=\{\beta \in \Sigma ^{+}:\beta |_{\mathfrak{a}}\neq 0\}.
\end{equation*}%
We write $\mathfrak{a}^{+}$ for the subset%
\begin{equation*}
\mathfrak{a}^{+}=\{H\in \mathfrak{a}:\alpha (H)>0\text{ for all }\alpha \in
\Phi ^{+}\}.
\end{equation*}%
The subsets $w(\overline{\mathfrak{a}^{+}})$ are disjoint for distinct $w\in
W$, the Weyl group of $U/K$, and $\mathfrak{a=}\bigcup_{w\in W}w(\overline{%
\mathfrak{a}^{+}})$. We will let 
\begin{equation*}
D=\dim \mathfrak{p}=\dim U/K\text{ and }r=\dim \mathfrak{a}.
\end{equation*}%
For notational convenience, we will put%
\begin{equation*}
\mathfrak{p}_{\ast }:=i\mathfrak{p}\text{ and }\mathfrak{a}_{\ast }^{+}:=-i%
\mathfrak{a}^{+}
\end{equation*}

The linear operator $Ad(k)$ maps $\mathfrak{p}_{\ast }\mathfrak{\rightarrow p%
}_{\ast }$ whenever $k\in K,$ and a function $f$ on $\mathfrak{p}_{\ast }$
is said to be $Ad(K)$-invariant if $f(Ad(k)X)=f(X)$ for all $X\in \mathfrak{p%
}_{\ast }$ and $k\in K$. Any continuous function defined on $\overline{%
\mathfrak{a}^{+}}$ has a unique $Ad(K)$-invariant extension to $\mathfrak{p}%
_{\ast }$.

The notation $\mu _{E}$ will denote Haar measure when $E=U,K,\mathfrak{a}%
_{\ast }$ or $\mathfrak{p}_{\ast }$ and will denote a $U$-invariant measure
on $U/K$. The measures will be normalized on $U$ and $K,$ and chosen
consistently so that the integration formulas,

\begin{equation}
\int_{\mathfrak{p}_{\ast }}f(X)d\mu _{\mathfrak{p}_{\ast }}(X)=\int_{%
\mathfrak{a}_{\ast }^{+}}\int_{K}f(Ad(k)H)\left\vert \prod_{\alpha \in \Phi
^{+}}\alpha (H)\right\vert d\mu _{K}(k)d\mu _{\mathfrak{a}_{\ast }}(H)
\label{IntegForm}
\end{equation}%
and%
\begin{equation}
\int_{U}fd\mu _{U}=\int_{U/K}\int_{K}f(uk)d\mu _{K}(k)d\mu _{U/K},
\label{QNorm}
\end{equation}%
hold for continuous functions $f$ of compact support (on the appropriate
domains). In particular, $\mu _{U/K}(S)=\mu _{U}(\pi ^{-1}(S))$ for Borel
sets\texttt{\ }$S$. We often omit the writing of $\mu _{E}$ if the
underlying space is clear.

As usual, by $L^{p}(E)$ we mean the functions defined on $E$ with $%
\left\Vert f\right\Vert _{L^{p}(E)}=\left( \int_{E}|f(X)|^{p}d\mu
_{E}(X)\right) ^{1/p}<\infty $. Functions on $U/K$ can be identified with
the right $K$-invariant functions on $U$, and these have the same $L^{p}$
norm.

\subsection{Multipliers on $L^{p}(\mathfrak{p}_{\ast })$}

The vector space $\mathfrak{p}_{\ast }$ can be viewed as a locally compact,
abelian group which is self-dual under the killing form $B(\cdot ,\cdot )$.
The Fourier transform of $f\in L^{2}(\mathfrak{p}_{\ast }\mathfrak{)}$ is
given by%
\begin{equation}
\widehat{f}(Y)=\int_{\mathfrak{p}_{\ast }}f(X)e^{-iB(X,Y)}d\mu _{\mathfrak{p}%
_{\ast }}(X)\text{ for }Y\in \mathfrak{p}_{\ast }  \label{FT}
\end{equation}%
and the Fourier inversion formula by%
\begin{equation}
\check{f}(X)=\int_{\mathfrak{p}_{\ast }}f(Y)e^{iB(X,Y)}d\mu _{\mathfrak{p}%
_{\ast }}(Y).  \label{Inversion}
\end{equation}

A bounded linear operator $T:$ $L^{p}(\mathfrak{p}_{\ast }\mathfrak{)}%
\rightarrow $ $L^{p}(\mathfrak{p}_{\ast }\mathfrak{)}$ is called an $L^{p}$%
\textit{\ multiplier }if there is a measurable function $m$ on $\mathfrak{p}%
_{\ast }$ so that for all $Y\in \mathfrak{p}_{\ast }$ and $f\in L^{2}\bigcap
L^{p}(\mathfrak{p}_{\ast }\mathfrak{)}$ we have $\widehat{Tf}(Y)=m(Y)%
\widehat{f}(Y)$. Often we write $T_{m}$ for $T$. We denote the operator norm
of $T_{m}$ by $\left\Vert T_{m}\right\Vert _{p,p}$.

Slightly abusing notation, we will also refer to $m$ as an $L^{p}$
multiplier and write $\left\Vert m\right\Vert _{p,p}$ for the operator norm
of $T_{m}$.

\subsection{Spherical multipliers on $L^{p}(U/K)$}

The left regular representation $\rho $ of $U,$ on the Hilbert space $%
L^{2}(U/K),$ provides a decomposition of $L^{2}(U/K)$ into an orthogonal
direct sum of invariant subspaces. The irreducible subrepresentations are
the class 1 representations of $(U,K)$, those with a one dimensional
subspace of $K$-fixed vectors. Let $\Lambda $ be the set of class 1 highest
weights. It is known (\cite[p.129]{H2}) that these are precisely of the form%
\begin{equation*}
\lambda =\sum_{j=1}^{r\text{ }}n_{j}\sigma _{j}
\end{equation*}%
where $n_{j}$ are non-negative integers for $j=1,...,r$ and $\{\sigma
_{1},...,\sigma _{r}\}$ is a suitable basis for $\mathfrak{a}$ (or more
formally, the dual of $\mathfrak{a}$, which we identify with $\mathfrak{a}$).

We will let $\{H_{1},...,H_{r}\}$ denote the dual basis of $\mathfrak{a}$
with respect to the Killing form $B$, i.e., 
\begin{equation*}
B(H_{j},X)=\sigma _{j}(X)\text{ for all }X\in \mathfrak{a}.
\end{equation*}%
With this notation, we have 
\begin{equation*}
\overline{\mathfrak{a}^{+}}=\{\sum_{j=1}^{r}n_{j}H_{j}:n_{j}\geq 0\}.
\end{equation*}%
Given $\lambda $ as above, by $H_{\lambda }\in \mathfrak{a}$ we mean the
element $H_{\lambda }=\sum_{j=1}^{r\text{ }}n_{j}H_{j}$. Conversely, when $%
Z=\sum n_{j}H_{j}\in \overline{\mathfrak{a}^{+}},$ we let $\lambda _{Z}$ be
the weight 
\begin{equation*}
\lambda _{Z}:=\sum n_{j}\sigma _{j}\text{ }.
\end{equation*}

By $d_{\lambda }$ we mean the degree of $\lambda \in \Lambda $. Having
chosen a $K$-fixed norm-one vector, $v_{\lambda },$ in the $\lambda $%
-representation space, we let $\phi _{\lambda }$ be the \textit{spherical
function} given by $\phi _{\lambda }(u)=\left\langle \rho (u)v_{\lambda
},v_{\lambda }\right\rangle $ for $u\in U.$ Since $v_{\lambda }$ is $K$%
-invariant, we can also view $\phi _{\lambda }$ as defined on $U/K$ in the
natural way.

If $f\in L^{2}(U/K)$, then we define 
\begin{equation*}
f\ast \phi _{\lambda }(\pi (x))=\int_{U}f(\pi (y))\phi _{\lambda
}(y^{-1}x)d\mu _{U}(y)\text{ for }x\in U\text{.}
\end{equation*}%
The Fourier series of $f$ is the formal sum 
\begin{equation*}
\sum_{\lambda \in \Lambda }d_{\lambda }f\ast \phi _{\lambda }.
\end{equation*}

A bounded linear operator $T:$ $L^{p}(U/K)$ $\rightarrow L^{p}(U/K)$ is
called a \textit{spherical multiplier} on $L^{p}(U/K)$ if there is a
function $\{m(\lambda )\}_{\lambda \in \Lambda }$ such that for $f\in
L^{2}\bigcap L^{p}(U/K)$, we have%
\begin{equation*}
T(f)=\sum_{\lambda \in \Lambda }d_{\lambda }m(\lambda )f\ast \phi _{\lambda
}.
\end{equation*}%
As before, we denote this operator by $T_{m}$ and also refer to the sequence 
$m$ as a \textit{spherical multiplier} on $L^{p}$. We write $\left\Vert
T_{m}\right\Vert _{p,p}$ or $\left\Vert m\right\Vert _{p,p}$ for its
operator norm.

Spherical multipliers are characterized by the property that they commute
with left translation by $U$.

We also view $m=\{m(\lambda )\}_{\lambda \in \Lambda }$ as being defined on
the \textquotedblleft integer-valued\textquotedblright\ elements of $%
\overline{\mathfrak{a}^{+}}$ and $\overline{\mathfrak{a}_{\ast }^{+}}$: If $%
Z\in \overline{\mathfrak{a}^{+}}$ has the form $Z=\sum n_{j}H_{j}$ with $%
n_{j}\in \mathbb{Z}^{+}$, we put 
\begin{equation*}
m(Z):=m(\lambda _{Z}).
\end{equation*}%
Similarly, if $Z=-\sum in_{j}H_{j}\in \overline{\mathfrak{a}_{\ast }^{+}}$
with $n_{j}\in \mathbb{Z}^{+}$, then we set $m(Z):=m(\lambda _{iZ})$.

For more details about symmetric spaces and their harmonic analysis we refer
the reader to \cite{H1} - \cite{OS} for example.

\section{\protect\bigskip Spherical multipliers on $U/K$ transfer to
multipliers on $\mathfrak{p}_{\ast }$}

\subsection{Notation}

We let $\exp :\mathfrak{u}\rightarrow U$ denote the exponential map. The
notation $\Pi _{1}$ will be used to denote the exponential map from $%
\mathfrak{p}_{\ast }$ to $U/K$, 
\begin{equation*}
\Pi _{1}(X)=\pi (\exp X)\text{ for }X\in \mathfrak{p}_{\ast }\text{.}
\end{equation*}%
More generally, for $t\geq 1$ we will let $\Pi _{t}:\mathfrak{p}_{\ast
}\rightarrow U/K$ be given by 
\begin{equation*}
\text{ }\Pi _{t}(X)=\Pi _{1}(X/t).
\end{equation*}

We will let $\Omega $ denote a convex neighbourhood of the identity in $%
\mathfrak{p}_{\ast }$ on which $\Pi _{1}$ is a diffeomorphism and let $J$ be
the Jacobian of $\Pi _{1}$,%
\begin{equation*}
J(Y)=\prod_{\alpha \in \Phi ^{+}}\frac{\sin \alpha (iH)}{\alpha (iH)}\text{
where }Y=Ad(k)H\text{, }H\in \mathfrak{a}_{\ast }\text{.}
\end{equation*}%
We will also assume that $\Omega $ is chosen suitably small that $J$ is
bounded away from $0$ and of course, $J$ is bounded by $1$.

We have the (change of variable) identity%
\begin{equation}
\int_{\Pi _{1}(\Omega )}f(\overline{x})d\mu _{U/K}(\overline{x}%
)=\int_{\Omega }f(\pi (\exp (Z))J(Z)d\mu _{\mathfrak{p}^{\ast
}}(Z)=\int_{\Omega }f(\Pi _{1}(Z))J(Z)d\mu _{\mathfrak{p}^{\ast }}(Z).
\label{CofV}
\end{equation}

Given $Z=\sum n_{j}H_{j}\in $ $\overline{\mathfrak{a}^{+}}$ and $t>0$, we
put 
\begin{equation}
\lbrack tZ]=\sum_{j}[tn_{j}]H_{j}  \label{tZ}
\end{equation}%
where $[tn_{j}]$ denotes the integer part of $tn_{j}$. As $[tn_{j}]\geq 0$, $%
[tZ]\in $ $\overline{\mathfrak{a}^{+}}$ and $\lambda _{\lbrack tZ]}=$ $\sum
[tn_{j}]\sigma _{j}$. We similarly understand $[itZ]$ when $Z\in \overline{%
\mathfrak{a}_{\ast }^{+}}$.

\subsection{A version of deLeeuw's theorem (2) for spherical multipliers}

First, we prove an analogue of the second part of deLeeuw's Theorem for the
pair $U/K$, $\mathfrak{p}_{\ast }$ that extends Stanton's Theorem 2.5 of 
\cite{St}.

\begin{theorem}
\label{2}Let $1<p<\infty $. Suppose $\{m_{t}\}_{t>0}$ is a family of
spherical multipliers on $L^{p}(U/K)$ with $\sup_{t}\left\Vert
m_{t}\right\Vert _{p,p}<\infty $. Assume we can define a continuous function 
$m$ on $\overline{\mathfrak{a}_{\ast }^{+}}$ by 
\begin{equation*}
m(Z)=\lim_{t\rightarrow \infty }m_{t}([itZ])\text{ for }Z\in \overline{%
\mathfrak{a}_{\ast }^{+}}.
\end{equation*}%
Then $m$ extends uniquely to a continuous $Ad(K)$-invariant function on $%
\mathfrak{p}_{\ast }$ and the linear map $T_{m}$ is a multiplier on $L^{p}(%
\mathfrak{p}_{\ast })$ satisfying $\left\Vert m\right\Vert _{p,p}\leq
C\sup_{t}\left\Vert m_{t}\right\Vert _{p,p},$ where $C$ is a constant
depending only on $p$.
\end{theorem}

For the proof we require the following Lemma that is a natural
generalization of \cite[Prop. 2.4]{St}. As it is technical, we will defer
its proof until after the conclusion of the proof of the Theorem.

\begin{lemma}
\label{Lemma}For $Z\in \mathfrak{a}_{\ast }^{+}$, $X,Y\in \mathfrak{p}_{\ast
}$, and $\lambda _{t}=\lambda _{\lbrack itZ]}$, we have%
\begin{equation*}
\lim_{t\rightarrow \infty }\phi _{\lambda _{t}}(\exp (Y/t)\exp
(X/t))=\int_{K}e^{iB(Z,Ad(k)(X+Y))}dk.
\end{equation*}
\end{lemma}

\begin{proof}[Proof of Theorem]
Throughout the proof, $c$ will denote a constant that may change.

First, assume that the family $\{m_{t}\}$ satisfies a decay condition,
namely, there are constants $C_{1},C_{2}$ such that 
\begin{equation}
\left\vert m_{t}([itZ])\right\vert \leq C_{1}\exp (-C_{2}\left\Vert
Z\right\Vert ^{2})  \label{decay}
\end{equation}
for all $Z\in \mathfrak{a}_{\ast }^{+}$ and large $t$.

The unique extension of $m$ to a continuous, $Ad(K)$-invariant function on $%
\mathfrak{p}_{\ast }$ is clear. Thus if we let $T=T_{m}$ denote the
corresponding linear operator, it will be enough to show that there is a
constant $C$ so that if%
\begin{equation*}
I:=\int_{\mathfrak{p}_{\ast }}Tf(X)g(X)dX,
\end{equation*}%
then 
\begin{equation*}
\left\vert I\right\vert \leq C\sup_{t}\left\Vert m_{t}\right\Vert
_{p,p}\left\Vert f\right\Vert _{L^{p}(\mathfrak{p}_{\ast })}\left\Vert
g\right\Vert _{L^{q}(\mathfrak{p}_{\ast })}
\end{equation*}%
whenever $f,g$ are $C^{\infty }$ functions on $\mathfrak{p}_{\ast }$ with
compact support and $q$ is the conjugate index to $p$.

Choose $t$ sufficiently large so that the set $(t^{-1}$supp $f)\bigcup
(t^{-1}$supp $g)$ is contained in $\Omega $ and hence $\Pi _{1}$ is a
diffeomorphism there. Define functions $f_{t}$ and $g_{t}$ on $U/K$ by 
\begin{equation*}
f_{t}(\pi (\exp (X))=f_{t}(\Pi _{1}(X))=f(tX)\text{ }
\end{equation*}%
and similarly for $g_{t}$. These are well defined because of the choice of $%
t $.

We will write $T_{t}$ for the spherical multiplier corresponding to $m_{t}$
and put 
\begin{equation*}
I_{t}=\int_{U/K}T_{t}(f_{t})(\overline{x})g_{t}(\overline{x})d\overline{x}.
\end{equation*}%
Obviously, we have%
\begin{equation*}
\left\vert I_{t}\right\vert \leq \sup_{t}\left\Vert m_{t}\right\Vert
_{p,p}\left\Vert f_{t}\right\Vert _{L^{p}(U/K)}\left\Vert g_{t}\right\Vert
_{L^{q}(U/K)}.
\end{equation*}

The first step is to calculate the $p$-norm of $f_{t}$. As $f_{t}$ is
supported on $\Pi _{1}(\Omega )$, the change of variables formula (\ref{CofV}%
) gives%
\begin{eqnarray*}
\left\Vert f_{t}\right\Vert _{L^{p}(U/K)}^{p} &=&\int_{U/K}|f_{t}(\overline{x%
})|^{p}d\overline{x}=\int_{\Pi _{1}(\Omega )}|f_{t}(\overline{x})|^{p}d%
\overline{x} \\
&=&\int_{\Omega }\left\vert f_{t}(\Pi _{1}(Y))\right\vert ^{p}J(Y)dY \\
&=&\int_{\mathfrak{p}_{\ast }}\left\vert f(tY)\right\vert
^{p}J(Y)dY=t^{-D}\int_{\mathfrak{p}_{\ast }}\left\vert f(Y)\right\vert
^{p}J(t^{-1}Y)dY.
\end{eqnarray*}%
As $|J(Y)|\leq 1$ for all $Y$, we see that $\left\Vert f_{t}\right\Vert
_{L^{p}(U/K)}\leq t^{-D/p}\left\Vert f\right\Vert _{L^{p}(\mathfrak{p}_{\ast
})}$.

Similarly, $\left\Vert g_{t}\right\Vert _{q}\leq t^{-D/q}\left\Vert
g\right\Vert _{q}$, so that 
\begin{equation*}
\left\vert I_{t}\right\vert \leq \sup_{t}\left\Vert m_{t}\right\Vert
_{p,p}t^{-D}\left\Vert f\right\Vert _{p}\left\Vert g\right\Vert _{q}.
\end{equation*}%
Thus it will be enough to prove that $\lim_{t\rightarrow \infty
}t^{D}I_{t}=cI$ for some constant $c$.

Now, 
\begin{equation*}
T_{t}(f_{t})=\sum_{\lambda \in \Lambda }d_{\lambda }m_{t}(\lambda )f_{t}\ast
\phi _{\lambda },
\end{equation*}%
therefore 
\begin{eqnarray*}
t^{D}I_{t} &=&t^{D}\int_{U/K}\sum_{\lambda \in \Lambda }m_{t}(\lambda
)d_{\lambda }f_{t}\ast \phi _{\lambda }(\overline{x})g_{t}(\overline{x})d%
\overline{x} \\
&=&t^{D}\int_{U}\int_{U}\sum_{\lambda \in \Lambda }d_{\lambda }m_{t}(\lambda
)f_{t}(\pi (y))\phi _{\lambda }(y^{-1}x)g_{t}(\pi (x))dydx.
\end{eqnarray*}%
For $t$ sufficiently large, change of variable arguments and the definitions
of $f_{t}$ and $g_{t}$ show that $t^{D}I_{t}$ equals%
\begin{equation*}
t^{-D}\int_{\mathfrak{p}_{\ast }}\int_{\mathfrak{p}_{\ast }}\sum_{\lambda
\in \Lambda }d_{\lambda }m_{t}(\lambda )f_{t}(\pi (\exp Y))\phi _{\lambda
}(\exp (-Y)\exp (X))g_{t}(\pi (\exp X))J(Y)J(X)d(Y)d(X)
\end{equation*}%
\begin{equation}
=t^{-D}\int_{\mathfrak{p}_{\ast }}\int_{\mathfrak{p}_{\ast }}\sum_{\lambda
\in \Lambda }d_{\lambda }m_{t}(\lambda )f(Y)\phi _{\lambda }(\exp (-Y/t)\exp
(X/t))g(X)J(Y/t)J(X/t)d(Y)d(X).  \label{F1}
\end{equation}

Recall that $\lambda \in \Lambda $ has the form $\lambda =\sum_{j=1}^{r\text{
}}n_{j}\sigma _{j}$ where $n_{j}\in \mathbb{Z}^{+}$, so that the sum over $%
\Lambda $ can be replaced by the sum over $\mathbb{Z}^{r+}$. This gives%
\begin{eqnarray*}
&&\sum_{\lambda \in \Lambda }d_{\lambda }m_{t}(\lambda )\phi _{\lambda
}(\exp (-Y/t)\exp (X/t)) \\
&=&\sum_{(n_{1},...,n_{r})\in \mathbb{Z}^{r+}}d_{\Sigma n_{j}\sigma
_{j}}m_{t}(\Sigma n_{j}\sigma _{j})\phi _{\Sigma n_{j}\sigma _{j}}(\exp
(-Y/t)\exp (X/t)) \\
&=&\sum_{\overrightarrow{n}\in \mathbb{Z}^{r+}}t^{r}\int_{\frac{n_{r}}{t}}^{%
\frac{n_{r}+1}{t}}\cdot \cdot \cdot \int_{\frac{n_{1}}{t}}^{\frac{n_{1}+1}{t}%
}m_{t}(\Sigma \lbrack tz_{j}]\sigma _{j})d_{\Sigma \lbrack tz_{j}]\sigma
_{j}}\phi _{\Sigma \lbrack tz_{j}]\sigma _{j}}(\exp (\frac{-Y}{t})\exp (%
\frac{X}{t}))dz_{1}...dz_{r} \\
&=&t^{r}\int_{\mathfrak{a}_{\ast }^{+}}m_{t}(\lambda _{\lbrack
itZ]})d_{\lambda _{\lbrack itZ]}}\phi _{\lambda _{\lbrack itZ]}}(\exp
(-Y/t)\exp (X/t))dZ
\end{eqnarray*}%
Combining this identity together with (\ref{F1}) and writing $\lambda _{t}$
for $\lambda _{\lbrack itZ]}$ gives%
\begin{equation*}
t^{D}I_{t}=t^{r-D}\int_{\mathfrak{p}_{\ast }}\int_{\mathfrak{p}_{\ast
}}\int_{\mathfrak{a}_{\ast }^{+}}m_{t}(\lambda _{t})d_{\lambda _{t}}\phi
_{\lambda _{t}}(\exp (\frac{-Y}{t})\exp (\frac{X}{t}))f(Y)g(X)J(\frac{Y}{t}%
)J(\frac{X}{t})dZdYdX.
\end{equation*}

The Weyl dimension formula states that%
\begin{equation*}
d_{\lambda _{\lbrack tiZ]}}=\prod_{\alpha \in \Sigma ^{+}}\frac{\left\langle
\alpha ,\lambda _{\lbrack itZ]}+\delta \right\rangle }{\left\langle \alpha
,\delta \right\rangle }
\end{equation*}%
where $\delta $ is half the sum of the positive roots. As $\lambda _{\lbrack
itZ]}$ is class 1, we have $\left\langle \alpha ,\lambda _{\lbrack
itZ]}\right\rangle =0$ if $\alpha \notin \Phi ^{+},$ thus writing $\{itZ\}$
for the `fractional' part of $itZ$ we have%
\begin{equation}
d_{\lambda _{\lbrack itZ]}}=t^{|\Phi ^{+}|}\prod_{\alpha \in \Phi ^{+}}\frac{%
\alpha (iZ-\{tiZ\}/t+H_{\delta }/t)}{\left\langle \alpha ,\delta
\right\rangle },  \label{Weyldim}
\end{equation}%
hence%
\begin{equation*}
\lim_{t\rightarrow \infty }\frac{d_{\lambda _{\lbrack tiZ]}}}{t^{|\Phi ^{+}|}%
}=\prod_{\alpha \in \Phi ^{+}}\frac{\alpha (iZ)}{\left\langle \alpha ,\delta
\right\rangle }.
\end{equation*}%
Moreover, $D=\dim \mathfrak{p}_{\ast }=\dim \mathfrak{a}$ $+|\Phi
^{+}|=r+|\Phi ^{+}|$, hence the Lemma implies that for $Z\in \mathfrak{a}%
_{\ast }^{+}$, 
\begin{eqnarray*}
&&\lim_{t\rightarrow \infty }t^{r-D}m_{t}(\lambda _{\lbrack itZ]})d_{\lambda
_{\lbrack itZ]}}\phi _{\lambda _{\lbrack itZ]}}(\exp (-Y/t)\exp
(X/t))J(Y/t)J(X/t) \\
&=&\prod_{\alpha \in \Phi ^{+}}\frac{\alpha (iZ)}{\left\langle \alpha
,\delta \right\rangle }m(Z)\int_{K}\exp \left( iB(Z,Ad(k)(X-Y)\right) dk%
\text{.}
\end{eqnarray*}

One can see from formula (\ref{Weyldim}) that there is some polynomial in
several variables, $P$, such that $t^{r-D}d_{\lambda _{\lbrack itZ]}}\leq
\left\vert P(Z)\right\vert $ for all $t$. Furthermore, $\left\vert \phi
_{\lambda }\right\vert $,$\left\vert J\right\vert \leq 1$, hence for $Z\in 
\mathfrak{a}_{\ast }^{+}$, 
\begin{equation*}
\left\vert t^{r-D}m_{t}(\lambda _{\lbrack itZ]})d_{\lambda _{t}}\phi
_{\lambda }(\exp (-Y/t)\exp (X/t))J(Y/t)J(X/t)\right\vert
\end{equation*}%
\begin{equation*}
\leq \left\vert P(Z)\right\vert C_{1}\exp (-C_{2}\left\Vert Z\right\Vert
^{2}),
\end{equation*}%
which is integrable over $\mathfrak{a}_{\ast }$. Since $f,g$ are continuous,
compactly supported functions, it follows from the Dominated convergence
theorem that 
\begin{equation}
t^{D}I_{t}\rightarrow \int_{\mathfrak{p}_{\ast }}\int_{\mathfrak{p}_{\ast
}}\int_{\mathfrak{a}_{\ast }^{+}}\prod_{\alpha \in \Phi ^{+}}\frac{\alpha
(iZ)}{\left\langle \alpha ,\delta \right\rangle }m(Z)%
\int_{K}e^{iB(Ad(k)Z,X-Y)}dkf(Y)g(X)dZdYdX.  \label{FormIt}
\end{equation}

Hence, it only remains to prove that the RHS of (\ref{FormIt}) is equal to $%
cI$ for some suitable constant $c$.

As $m$ is $Ad(K)$-invariant and $\alpha (iZ)\geq 0,$ the integration formula
(\ref{IntegForm}) implies%
\begin{equation*}
\int_{\mathfrak{a}_{\ast }^{+}}\int_{K}m(Z)e^{iB(Ad(k)Z,X-Y)}dk\prod_{\alpha
\in \Phi ^{+}}\alpha (iZ)dZ=\int_{\mathfrak{p}_{\ast }}m(Z)e^{iB(Z,X-Y)}dZ.
\end{equation*}%
Thus the RHS\ of (\ref{FormIt}) is equal to%
\begin{eqnarray}
&&c\int_{\mathfrak{p}_{\ast }}\int_{\mathfrak{p}_{\ast }}\int_{\mathfrak{p}%
_{\ast }}m(Z)e^{iB(Z,X-Y)}f(Y)g(X)dZdYdX  \notag \\
&=&c\int_{\mathfrak{p}_{\ast }}\int_{\mathfrak{p}_{\ast }}\int_{\mathfrak{p}%
_{\ast }}m(Z)e^{iB(Z,-Y)}f(Y)e^{iB(Z,X)}g(X)dYdZdX  \label{F2}
\end{eqnarray}%
where $c=\prod_{\alpha \in \Phi ^{+}}\left\langle \alpha ,\delta
\right\rangle ^{-1}$ and Fubini's theorem is justified by the exponential
decay in the function $m$. The Fourier transform and inversion formulas (see
(\ref{FT}), (\ref{Inversion})) simplify (\ref{F2}) to%
\begin{eqnarray*}
c\int_{\mathfrak{p}_{\ast }}\int_{\mathfrak{p}_{\ast }}m(Z)\widehat{f}%
(Z)e^{iB(Z,X)}g(X)dZdX &=&c\int_{\mathfrak{p}_{\ast }}\int_{\mathfrak{p}%
_{\ast }}\widehat{Tf}(Z)e^{iB(Z,X)}g(X)dZdX \\
&=&c\int_{\mathfrak{p}_{\ast }}Tf(X)g(X)dX.
\end{eqnarray*}%
As this equals $cI,$ the proof that $\left\Vert m\right\Vert _{p,p}\leq
C\sup_{t}\left\Vert m_{t}\right\Vert _{p,p}$ for a suitable constant $C$ is
complete under the additional decay assumption.

In the general case, for each $\varepsilon >0$ and $t$ large, let $%
n_{t,\varepsilon }(\lambda )=\exp (-\varepsilon \left\Vert \lambda
\right\Vert ^{2}/t^{2})$. The rapid decay of the function $z\rightarrow \exp
(-\delta \left\Vert z\right\Vert ^{2})$ for $z\in \mathbb{R}^{n}$ and any $%
\delta >0$, together with all its derivatives, allows one to use the
Hormander-Mihlin style central multiplier theorem for $L^{p}(U)$ (c.f., \cite%
{We}) to deduce that the functions $n_{t,\varepsilon }$ are $L^{p}$
spherical multipliers on $U/K$ and, furthermore, that their operator norms
are bounded by a constant $C_{1}$ that depends only on $p$.

It follows that the functions $m_{t,\varepsilon }(\lambda )=m_{t}(\lambda
)n_{t,\varepsilon }(\lambda )$ satisfy $\sup_{t}\left\Vert m_{t,\varepsilon
}\right\Vert _{p,p}\leq C_{1}\sup_{t}\left\Vert m_{t}\right\Vert _{p,p},$ as
well as the decay condition (\ref{decay}). By the first part of the proof, 
\begin{equation*}
m_{\varepsilon }(Z)=\lim_{t\rightarrow \infty }m_{t,\varepsilon }([itZ])
\end{equation*}%
is an $L^{p}(\mathfrak{p}_{\ast })$ multiplier with operator norm 
\begin{equation*}
\left\Vert m_{\varepsilon }\right\Vert _{p,p}\leq C\sup_{t}\left\Vert
m_{t,\varepsilon }\right\Vert _{p,p}\leq CC_{1}\sup_{t}\left\Vert
m_{t}\right\Vert _{p,p}.
\end{equation*}%
Letting $\varepsilon \rightarrow 0$, it follows that $m$ is also an $L^{p}(%
\mathfrak{p}_{\ast })$ multiplier with norm also bounded by $%
CC_{1}\sup_{t}\left\Vert m_{t}\right\Vert _{p,p}.$
\end{proof}

We turn now to proving Lemma \ref{Lemma}.

\begin{proof}[Proof of Lemma]
Put $\mathfrak{g}^{\mathbb{C}}=\mathfrak{k}^{\mathbb{C}}\oplus \mathfrak{a}^{%
\mathbb{C}}\oplus \mathfrak{n}^{\mathbb{C}}$ and let $\mathcal{P}$ denote
the projection onto $\mathfrak{a}^{\mathbb{C}}$. Let $G^{\mathbb{C}}$ be the
complexification of $U$ and denote by $G_{0}$ its subgroup with Lie algebra $%
\mathfrak{g}_{0}$. Then $G_{0}$ has Iwasawa decomposition $G_{0}=KAN$. Let $%
\mathcal{H}:G_{0}\rightarrow \mathfrak{a}$ be given by the rule $x=k\exp 
\mathcal{H}(x)n$ and continue it analytically to a neighbourhood of $e$ in $%
G^{\mathbb{C}}$. It is shown in \cite[Lemma 2.2, Prop. 2.3]{St} that for $s$
small and $Z$ in a suitable neighbourhood of $0$ in $\mathfrak{g}^{\mathbb{C}%
}$, 
\begin{equation*}
\mathcal{H}(\exp sZ)=s\mathcal{P}(Z)+O(s^{2})
\end{equation*}%
and also that if $X\in \mathfrak{g}^{\mathbb{C}}$ has sufficiently small
norm, then 
\begin{equation*}
\phi _{\lambda }(\exp X)=\int_{K}e^{\lambda (\mathcal{H}(\exp Ad(k)(X))}dk.
\end{equation*}

By the Hausdorff-Campbell formula, 
\begin{equation*}
\exp (-Y/t)\exp (X/t)=\exp ((X-Y)/t+W_{t}(X,Y))
\end{equation*}%
where $\left\Vert W_{t}(X,Y)\right\Vert \leq O(1/t^{2})$. Putting these
facts together and recalling that $\lambda _{t}=\lambda _{\lbrack itZ]},$ we
see that%
\begin{eqnarray*}
\phi _{\lambda _{t}}\left( \exp (\frac{-Y}{t})\exp (\frac{X}{t})\right) 
&=&\int_{K}\exp \left( \lambda _{t}(\mathcal{H}(\exp
Ad(k)(t^{-1}(X-Y)+W_{t}(X,Y)))\right) dk \\
&=&\int_{K}\exp \left( B([itZ],\mathcal{H}(\exp \frac{1}{t}%
Ad(k)(X-Y+tW_{t}(X,Y))))\right) dk \\
&=&\int_{K}\exp \left( B([itZ],\frac{1}{t}\mathcal{P}(Ad(k)(X-Y+tW_{t})+O(%
\frac{1}{t^{2}}))\right) dk.
\end{eqnarray*}

Writing $[itZ]=itZ-\{itZ\}$, this becomes%
\begin{equation*}
\int_{K}\exp \left( iB(Z,\mathcal{P}(Ad(k)(X-Y))+O(\frac{1}{t}))-B(\{itZ\},%
\frac{1}{t}\mathcal{P}(Ad(k)(X-Y))+O(\frac{1}{t^{2}}))\right) dk.
\end{equation*}%
As $\left\Vert \{itZ\}\right\Vert $ is bounded (over all $Z$ and $t)$, 
\begin{equation*}
\left\vert B\left( \{itZ\},\frac{1}{t}\mathcal{P}(Ad(k)(X-Y)+O(\frac{1}{t^{2}%
}))\right) \right\vert \leq O(\frac{1}{t})
\end{equation*}%
uniformly over $k\in K$, thus another application of the Dominated
convergence theorem implies that as $t\rightarrow \infty $%
\begin{equation*}
\phi _{\lambda _{\lbrack itZ]}}\left( \exp (\frac{-Y}{t})\exp (\frac{X}{t}%
)\right) \rightarrow \int_{K}e^{iB(Z,\mathcal{P}(Ad(k)(X-Y))}dk=%
\int_{K}e^{iB(Z,Ad(k)(X-Y))}dk,
\end{equation*}%
where the final equality was shown in the Proof of Prop 2.4 in \cite{St}.
\end{proof}

A special case of the theorem is \cite[Thm. 2.5]{St}.

\begin{corollary}
Suppose $m$ is an $Ad(K)$-invariant, continuous, bounded function on $%
\mathfrak{p}_{\ast }$. For $t>0,$ define $m_{t}(Z)=m(Z/t)$ for $Z\in 
\overline{\mathfrak{a}_{\ast }^{+}}$ and assume that $\{m_{t}\}_{t>0}$ is a
family of spherical multipliers on $L^{p}(U/K)$ with uniformly bounded
operator norms. Then $T_{m}$ is a multiplier on $L^{p}(\mathfrak{p}_{\ast })$%
.
\end{corollary}

\begin{proof}
It is enough to note that a continuity argument implies 
\begin{equation*}
\lim_{r\rightarrow \infty }m_{t}([tZ])=\lim_{t}m\left( [tZ]/t\right) =m(Z)%
\text{ for }Z\in \overline{\mathfrak{a}_{\ast }^{+}}.
\end{equation*}%
Then call upon the theorem.
\end{proof}

\section{Multipliers on $\mathfrak{p}_{\ast }$ transfer to spherical
multipliers on $U/K$}

In this section we prove an analogue of the first part of deLeeuw's theorem
that is a direct converse of the analogue of the second part that we proved
in the previous section, for a restricted class of multipliers.

We continue to assume that $\Omega $ is a convex neighbourhood of the
identity in $\mathfrak{p}_{\ast }$ on which $\Pi _{1}$ is a diffeomorphism
and that the Jacobian of $\Pi _{1},$ $J,$ is bounded away from zero. We fix
a convex, symmetric neighbourhood of the identity in $\mathfrak{p}_{\ast }$, 
$\mathcal{O}\subseteq \Omega ,$ that is relatively compact.

We remind the reader that the notation $[tZ]$ was defined in (\ref{tZ}).

\begin{theorem}
Let $1<p<\infty $. Assume $\xi $ $\in L^{1}(\mathfrak{p}_{\ast })$ is an $%
Ad(K)$-invariant function supported on the neighbourhood $\mathcal{O}$.
There are constants $C_{1},C_{2}>0$ and $L^{p}$ spherical multipliers on $U/K
$, $\{m_{t}(\lambda )\}_{\lambda \in \Lambda },$ such that 
\begin{equation*}
\lim_{t\rightarrow \infty }m_{t}(\lambda _{\lbrack itZ]})=\widehat{\xi }(Z)%
\text{ for all }Z\in \overline{\mathfrak{a}_{\ast }^{+}}.
\end{equation*}%
and%
\begin{equation*}
C_{1}\limsup_{t\rightarrow \infty }\left\Vert m_{t}\right\Vert _{p,p}\leq
\left\Vert \xi \right\Vert _{p,p}\leq C_{2}\limsup_{t\rightarrow \infty
}\left\Vert m_{t}\right\Vert _{p,p},
\end{equation*}%
where we view $\xi $ as a linear operator on $L^{p}$ with the action given
by convolution.
\end{theorem}

\begin{proof}
The approach we take to this proof is motivated by \cite{DGu2} and \cite{DGR}%
. Throughout the proof $c$ will denote a constant that may change. Without
loss of generality we will assume $t\geq 1$.

Let $q$ be the conjugate index to $p$. Given $F\in L^{p}(U/K)$ and $G\in
L^{q}(U/K),$ we define $F_{t}$ and $G_{t}$ on $\mathfrak{p}_{\ast }$ by%
\begin{equation*}
F_{t}(Z)=\left\{ 
\begin{array}{cc}
t^{-D/p}F(\Pi _{t}(Z)) & \text{if }Z\in \mathcal{O} \\ 
0 & \text{else}%
\end{array}%
\right.
\end{equation*}%
and 
\begin{equation*}
G_{t}(Z)=\left\{ 
\begin{array}{cc}
t^{-D/q}G(\Pi _{t}(Z)) & \text{if }Z\in \mathcal{O}+\mathcal{O} \\ 
0 & \text{else}%
\end{array}%
\right. .
\end{equation*}

Note that $\mathcal{O}/t\subseteq \mathcal{O}$, hence, with $c>0$ chosen
such that $J\geq 1/c$ on $\mathcal{O}$, we have%
\begin{eqnarray*}
\left\Vert F_{t}\right\Vert _{L^{p}(\mathfrak{p}_{\ast })}^{p}
&=&t^{-D}\int_{\mathcal{O}}\left\vert F(\Pi _{t}(X))\right\vert ^{p}dX \\
&\leq &ct^{-D}\int_{\mathcal{O}}\left\vert F(\Pi _{t}(X))\right\vert
^{p}J(t^{-1}X)dX.
\end{eqnarray*}%
Making the change of variable $X=tW$ and simplifying gives%
\begin{eqnarray*}
\left\Vert F_{t}\right\Vert _{L^{p}(\mathfrak{p}_{\ast })} &\leq &c\left(
\int_{\mathcal{O}/t}\left\vert F(\Pi _{1}(W))\right\vert ^{p}J(W)dW\right)
^{1/p} \\
&=&c\left( \int_{\Pi _{t}(\mathcal{O})}\left\vert F(\overline{x})\right\vert
^{p}d\overline{x}\right) ^{1/p}=c\left\Vert F|_{\Pi _{t}(\mathcal{O}%
)}\right\Vert _{L^{p}(U/K)}.
\end{eqnarray*}%
Similarly, we have 
\begin{equation*}
\left\Vert G_{t}\right\Vert _{L^{q}(\mathfrak{p}_{\ast })}\leq c\left\Vert
G|_{\Pi _{t}(\mathcal{O}+\mathcal{O})}\right\Vert _{L^{q}(U/K)}.
\end{equation*}

For $f\in L^{p}(\mathfrak{p}_{\ast })$ and $g\in L^{q}(\mathfrak{p}_{\ast })$%
, we will define the linear action 
\begin{equation*}
<f,g>_{\mathfrak{p}_{\ast }}=\int_{\mathfrak{p}_{\ast }}f(X)g(X)d\mu _{%
\mathfrak{p}_{\ast }}(X)
\end{equation*}%
and similarly for $F,G$ defined on $U/K$. The above computations show that 
\begin{eqnarray*}
\left\vert <F_{t}\ast \xi ,G_{t}>_{\mathfrak{p}_{\ast }}\right\vert  &\leq
&\left\Vert F_{t}\ast \xi \right\Vert _{L^{p}(\mathfrak{p}_{\ast
})}\left\Vert G_{t}\right\Vert _{L^{q}(\mathfrak{p}_{\ast })}\leq \left\Vert
\xi \right\Vert _{p,p}\left\Vert F_{t}\right\Vert _{p}\left\Vert
G_{t}\right\Vert _{q} \\
&\leq &c\left\Vert \xi \right\Vert _{p,p}\left\Vert F|_{\Pi _{t}(\mathcal{O}%
)}\right\Vert _{L^{p}(U/K)}\left\Vert G|_{\Pi _{t}(\mathcal{O}+\mathcal{O}%
)}\right\Vert _{L^{q}(U/K)}.
\end{eqnarray*}%
This proves we can define a bounded linear operator $V_{t}:$ $%
L^{p}(U/K)\rightarrow $ $L^{p}(U/K)$ by the rule that for each $F\in
L^{p}(U/K)$, the linear function 
\begin{equation*}
V_{t}(F):L^{q}(U/K)\rightarrow \mathbb{C}
\end{equation*}%
is given by 
\begin{equation*}
<V_{t}(F),G>_{U/K}\text{ }=\text{ }<F_{t}\ast \xi ,G_{t}>_{\mathfrak{p}%
_{\ast }}\text{ for all }G\in L^{q}(U/K).
\end{equation*}

As $V_{t}$ need not commute with left translation by elements of $U$, we
consider the linear map $T_{t}:$ $L^{p}(U/K)\rightarrow $ $L^{p}(U/K)$ given
by 
\begin{equation*}
<T_{t}F,G>_{U/K}\text{ }=t^{D}\int_{U}<V_{t}(\rho _{y}(F)),\rho
_{y}(G)>_{U/K}dy\text{ for }G\in L^{q}(U/K)
\end{equation*}%
where $\rho _{y}(F)(\overline{x})=F(y^{-1}\overline{x})=F(\pi (y^{-1}x))$
for any $y\in U$. We remark that $\rho _{yz}=\rho _{y}\rho _{z}$. An
application of Holder's inequality shows that for any $t$, 
\begin{eqnarray*}
&&\left\vert <T_{t}F,G>\right\vert \leq ct^{D}\int_{U}\left\Vert \xi
\right\Vert _{p,p}\left\Vert (\rho _{y}F)|_{\Pi _{t}(\mathcal{O}%
)}\right\Vert _{L^{p}(U/K)}\left\Vert (\rho _{y}G)|_{\Pi _{t}(\mathcal{O}+%
\mathcal{O})}\right\Vert _{L^{q}(U/K)}dy \\
&\leq &ct^{D}\left\Vert \xi \right\Vert _{p,p}\left( \int_{U}\left\Vert
(\rho _{y}F)|_{\Pi _{t}(\mathcal{O})}\right\Vert _{p}^{p}dy\right)
^{1/p}\left( \int_{U}\left\Vert (\rho _{y}G)|_{\Pi _{t}(\mathcal{O}+\mathcal{%
O})}\right\Vert _{q}^{q}dy\right) ^{1/q}.
\end{eqnarray*}%
Fubini's theorem gives 
\begin{eqnarray*}
\int_{U}\left\Vert (\rho _{y}F)|_{\Pi _{t}(\mathcal{O})}\right\Vert
_{p}^{p}dy &=&\int_{U/K}\int_{U}\chi _{\Pi _{t}(\mathcal{O})}(\overline{x}%
)\left\vert F\circ \pi (y^{-1}x)\right\vert ^{p}dyd\overline{x} \\
&=&\int_{U}\left( \int_{U}\left\vert F(\pi (y^{-1}x))\right\vert
^{p}dy\right) \chi _{\Pi _{t}(\mathcal{O})}(\pi (x))dx \\
&=&\int_{U}\left( \int_{U}\left\vert F(\pi (y^{-1}))\right\vert
^{p}dy\right) \chi _{\Pi _{t}(\mathcal{O})}(\pi (x))dx
\end{eqnarray*}%
\texttt{\ }after replacing $y$ by $xy$. Thus 
\begin{eqnarray*}
\int_{U}\left\Vert (\rho _{y}F)|_{\Pi _{t}(\mathcal{O})}\right\Vert
_{p}^{p}dy &=&\int_{U}\left( \left\Vert F\right\Vert
_{L^{p}(U/K)}^{p}\right) \chi _{\Pi _{t}(\mathcal{O})}(\pi (x))dx \\
&=&\mu _{U/K}(\Pi _{t}(\mathcal{O}))\left\Vert F\right\Vert
_{L^{p}(U/K)}^{p}.
\end{eqnarray*}%
\texttt{\ }Similarly, 
\begin{equation*}
\int_{U}\left\Vert \rho _{y}G|_{\Pi _{t}(\mathcal{O}+\mathcal{O}%
)}\right\Vert _{q}^{q}dy=\mu _{U/K}(\Pi _{t}(\mathcal{O}+\mathcal{O}%
))\left\Vert G\right\Vert _{L^{q}(U/K)}^{q}.
\end{equation*}

Since $\Pi _{t}(\mathcal{O})\subseteq \Pi _{1}(\Omega )\subseteq U/K,$ 
\begin{equation*}
\mu _{U/K}(\Pi _{t}(\mathcal{O}))=\int_{U/K}\chi _{\Pi _{t}(\mathcal{O})}(%
\overline{u})d\overline{u}=\int_{\Pi _{1}(\Omega )}\chi _{\Pi _{t}(\mathcal{O%
})}(\overline{u})d\overline{u}.
\end{equation*}%
More change of variables arguments and the fact that $\Pi _{1}$ is a
diffeomorphism on $\Omega $ means that 
\begin{eqnarray*}
\mu _{U/K}(\Pi _{t}(\mathcal{O})) &=&\int_{\Omega }\chi _{\Pi _{t}(\mathcal{O%
})}(\Pi _{1}(Z))J(Z)\,dZ \\
&=&t^{-D}\int_{t\Omega }\chi _{\Pi _{t}(\mathcal{O})}(\Pi
_{t}(Z))J(t^{-1}Z)dZ\leq t^{-D}\mu _{\mathfrak{p}_{\ast }}(\mathcal{O})
\end{eqnarray*}%
and this is finite as $\mathcal{O}$ is pre-compact. Similarly, 
\begin{equation*}
\mu _{U/K}(\Pi _{t}(\mathcal{O}+\mathcal{O}))\leq t^{-D}\mu _{\mathfrak{p}%
_{\ast }}(\mathcal{O}+\mathcal{O})<\infty .
\end{equation*}%
Thus%
\begin{eqnarray*}
&<&T_{t}F,G>_{U/K} \\
&\leq &ct^{D}\left\Vert \xi \right\Vert _{p,p}\left( t^{-D}\mu (\Pi _{t}(%
\mathcal{O}))\right) ^{1/p}\left\Vert F\right\Vert _{p}\left( t^{-D}\mu (\Pi
_{t}(\mathcal{O}+\mathcal{O}))\right) ^{1/q}\left\Vert G\right\Vert _{q} \\
&\leq &c\left\Vert \xi \right\Vert _{p,p}\left\Vert F\right\Vert
_{p}\left\Vert G\right\Vert _{q}
\end{eqnarray*}%
(for a different constant on the third line) and therefore the Riesz
Representation theorem implies%
\begin{equation*}
\left\Vert T_{t}F\right\Vert _{L^{p}(U/K)}\leq c\left\Vert \xi \right\Vert
_{p,p}\left\Vert F\right\Vert _{L^{p}(U/K)}.
\end{equation*}

Next, we check that $T_{t}$ commutes with translation on $U$. Indeed,
suppose $F,G$ are continuous functions on $U/K$. For each $v\in U$,%
\begin{eqnarray*}
&<&T_{t}(\rho _{v}F),G>_{U/K}\text{ }=t^{D}\int_{U}<V_{t}\rho _{y}(\rho
_{v}(F)),\rho _{y}(G)>dy \\
&=&t^{D}\int_{U}<V_{t}\rho _{y}(F),\rho _{yv^{-1}}(G)>dy \\
&=&<T_{t}F,\rho _{v^{-1}}G>_{U/K}\text{ }=\text{ }<\rho _{v}T_{t}F,G>_{U/K},
\end{eqnarray*}%
so $T_{t}(\rho _{v}F)=\rho _{v}(T_{t}F)$ for all $F\in C(U/K).$ A denseness
argument implies $T_{t}$ commutes with $\rho _{v}$ for all $v\in U$.

These facts establish that for each $t$, $T_{t}$ is a spherical multiplier
satisfying $\left\Vert T_{t}\right\Vert _{p,p}\leq c\left\Vert \xi
\right\Vert _{p,p}$.

Now define a multiplier $S_{t}$ by $S_{t}(F)=\frac{1}{\mu (\mathcal{O)}}%
T_{t}(F)$; of course $\left\Vert S_{t}\right\Vert _{p,p}\leq c\left\Vert \xi
\right\Vert _{p,p}$ for a new constant $c$. Let $m_{t}$ be the associated
multiplier sequence, i.e., $S_{t}=S_{m_{t}}$.\texttt{\ }

We will next prove that $\lim_{t\rightarrow \infty }m_{t}(\lambda _{\lbrack
itZ]})=\widehat{\xi }(Z)$. The final statement of the theorem, $\left\Vert
\xi \right\Vert _{p,p}\leq C_{2}\limsup_{t\rightarrow \infty }\left\Vert
m_{t}\right\Vert _{p,p},$ will then follow from Theorem \ref{2}. To do this,
we will find a different formulation of $T_{t}$. Applying the definitions of 
$T_{t},V_{t}$ and $\Pi _{t}$ gives%
\begin{eqnarray*}
&<&T_{t}F,G>_{U/K}\text{ }=t^{D}\int_{U}<V_{t}\rho _{y}(F),\rho
_{y}(G)>_{U/K}dy \\
&=&t^{D}\int_{U}<(\rho _{y}F)_{t}\ast \xi ,(\rho _{y}G)_{t}>_{\mathfrak{p}%
_{\ast }}dy=t^{D}\int_{U}\int_{\mathfrak{p}_{\ast }}\left( (\rho
_{y}F)_{t}\ast \xi \right) (X)(\rho _{y}G)_{t}(X)dXdy \\
&=&t^{D}\int_{U}\int_{\mathfrak{p}_{\ast }}\int_{\mathfrak{p}_{\ast }}(\rho
_{y}F)_{t}(W)\xi (-W+X)(\rho _{y}G)_{t}(X)dWdXdy \\
&=&\int_{U}\int_{\mathcal{O}+\mathcal{O}}\int_{\mathcal{O}}(\rho _{y}F)(\Pi
_{t}(W))\xi (-W+X)(\rho _{y}G)(\Pi _{t}(X))dWdXdy \\
&=&\int_{U}\int_{\mathcal{O}+\mathcal{O}}\int_{\mathcal{O}}F(\pi (y^{-1}\exp
W/t))\xi (-W+X)G(\pi (y^{-1}\exp X/t))dWdXdy.
\end{eqnarray*}%
After doing the change of variable $y\rightarrow (\exp X/t)y$ and inversion (%
$y\rightarrow y^{-1})$ we obtain%
\begin{equation*}
<T_{t}F,G>=t^{D}\int_{U}\int_{\mathcal{O}+\mathcal{O}}\int_{\mathcal{O}%
}F(\pi (y\exp (\frac{-X}{t})\exp (\frac{W}{t}))\xi (-W+X)G(\pi (y))dWdXdy.
\end{equation*}%
As $F$ is continuous, this proves that 
\begin{equation*}
T_{t}F(\overline{y})=\int_{\mathcal{O}+\mathcal{O}}\int_{\mathcal{O}}F(\pi
(y\exp (-X/t)\exp (W/t))\xi (-W+X)dWdX.
\end{equation*}

Changing the order of integration, noting that for a given $W\in \mathcal{O}$
the integral over the variable $X$ is limited to $X\in W+\mathcal{O},$ and
then doing the change of variable $X\rightarrow X+W$ gives%
\begin{equation}
T_{t}F(\overline{y})=\int_{\mathcal{O}}\int_{\mathcal{O}}F(\pi (y\exp
(-(X+W)/t)\exp (W/t))\xi (X)dXdW.  \label{TrFormula}
\end{equation}

Recall the Fourier series formula%
\begin{equation*}
S_{t}F=\sum_{\lambda \in \Lambda }d_{\lambda }m_{t}(\lambda )F\ast \phi
_{\lambda }.
\end{equation*}%
\texttt{\ }For $Z\in \mathfrak{a}_{\ast }^{+}$, $\lambda _{t}=\lambda
_{\lbrack itZ]}$ and $F=\phi _{\lambda _{t}}$ we have 
\begin{equation*}
S_{t}\phi _{\lambda _{t}}(e)=\sum_{\lambda \in \Lambda }d_{\lambda
}m_{t}(\lambda )\phi _{\lambda _{t}}\ast \phi _{\lambda }(e).
\end{equation*}%
As $\phi _{\lambda }\ast \phi _{\sigma }(e)=0$ if $\lambda \neq \sigma $ and
equals $1/d_{\lambda }$ else, it follows that $S_{t}\phi _{\lambda
_{t}}(e)=m_{t}(\lambda _{\lbrack itZ]})$.

According to Lemma \ref{Lemma}, 
\begin{equation*}
\lim_{t\rightarrow \infty }\phi _{\lambda _{t}}(\exp (-(X+W)/t)\exp (W/t))=%
\mathcal{J}(Z,-X).
\end{equation*}%
where 
\begin{equation*}
\mathcal{J}(Z,X)=\int_{K}e^{iB(Z,Ad(k)X)}dk.
\end{equation*}%
Since $\left\vert \phi _{\lambda }\right\vert \leq 1$ and $\mathcal{O}$ is
relatively compact, this identity, together with (\ref{TrFormula}) and the
fact that $\xi $ is supported on $\mathcal{O}$ implies%
\begin{equation*}
\lim_{t\rightarrow \infty }S_{t}\phi _{\lambda _{t}}(e)=\frac{1}{\mu _{%
\mathfrak{p}_{\ast }}(\mathcal{O})}\int_{\mathcal{O}}\int_{\mathcal{O}}%
\mathcal{J}(Z,-X)\xi (X)dWdX=\int_{\mathfrak{p}_{\ast }}\mathcal{J}(Z,-X)\xi
(X)dX.
\end{equation*}%
Apply the integration formula (\ref{IntegForm}), to get 
\begin{equation*}
\lim_{t\rightarrow \infty }S_{t}\phi _{\lambda _{t}}(e)=\int_{\mathfrak{a}%
_{\ast }^{+}}\int_{K}\xi (Ad(k)H)\mathcal{J}(Z,-Ad(k)H)\left\vert
\prod_{\alpha \in \Phi ^{+}}\alpha (H)\right\vert dkdH.
\end{equation*}%
Finally, using $K$-invariance and symmetry properties, we obtain%
\begin{equation*}
\lim_{t\rightarrow \infty }m_{t}(\lambda _{\lbrack itZ]})=\lim_{t\rightarrow
\infty }S_{t}\phi _{\lambda _{t}}(e)=\int_{\mathfrak{a}_{\ast }^{+}}\xi (H)%
\mathcal{J}(-H,Z)\left\vert \prod_{\alpha \in \Phi ^{+}}\alpha
(H)\right\vert dH
\end{equation*}%
\begin{eqnarray*}
&=&\int_{\mathfrak{a}_{\ast }^{+}}\int_{K}\xi
(Ad(k)H)e^{-iB(Z,Ad(k)H)}dk\left\vert \prod_{\alpha \in \Phi ^{+}}\alpha
(H)\right\vert dkdH \\
&=&\int_{\mathfrak{p}_{\ast }}e^{-iB(Y,Z)}\xi (Y)dY=\widehat{\xi }(Z)\text{ }
\end{eqnarray*}%
for $Z\in \overline{\mathfrak{a}_{\ast }^{+}},$ as we desired to show.
\end{proof}

\end{document}